\documentclass[11pt,reqno,letterpaper]{amsart}
\usepackage{color}
\usepackage[colorlinks=true, allcolors=blue,backref=page]{hyperref}
\usepackage{amsmath, amssymb, amsthm}
\usepackage{mathrsfs}
\usepackage{mathtools}
\usepackage[noabbrev,capitalize,nameinlink]{cleveref}
\crefname{equation}{}{}
\usepackage{fullpage}
\usepackage[noadjust]{cite}
\usepackage{graphics}
\usepackage{pifont}
\usepackage{tikz}
\usepackage{bbm}
\usepackage[T1]{fontenc}

\usetikzlibrary{arrows.meta}

\usepackage{environ}
\usepackage{framed}
\usepackage{url}
\usepackage[linesnumbered,ruled,vlined]{algorithm2e}
\usepackage[noend]{algpseudocode}
\usepackage[labelfont=bf]{caption}
\usepackage{cite}
\usepackage{framed}
\usepackage[framemethod=tikz]{mdframed}
\usepackage{appendix}
\usepackage{graphicx}
\usepackage[textsize=tiny]{todonotes}
\usepackage{tcolorbox}
\usepackage{enumerate}
\allowdisplaybreaks[1]
\usepackage{stmaryrd}
\usepackage{colonequals}
\usepackage[margin=1in]{geometry}

\usepackage[shortlabels]{enumitem}
\crefformat{enumi}{#2#1#3}
\crefrangeformat{enumi}{#3#1#4 to~#5#2#6}
\crefmultiformat{enumi}{#2#1#3}
{ and~#2#1#3}{, #2#1#3}{ and~#2#1#3}

\DeclareSymbolFont{symbolsC}{U}{pxsyc}{m}{n}
\SetSymbolFont{symbolsC}{bold}{U}{pxsyc}{bx}{n}
\DeclareFontSubstitution{U}{pxsyc}{m}{n}
\DeclareMathSymbol{\medcircle}{\mathbin}{symbolsC}{7}

\crefname{algocf}{Algorithm}{Algorithms}

\crefname{equation}{}{} 
\AtBeginEnvironment{appendices}{\crefalias{section}{appendix}} 

\usepackage[color,final]{showkeys} 

\colorlet{refkey}{orange!20}
\colorlet{labelkey}{blue!30}

\crefname{algocf}{Algorithm}{Algorithms}

\numberwithin{equation}{section}
\newtheorem{theorem}{Theorem}[section]
\newtheorem{proposition}[theorem]{Proposition}
\newtheorem{lemma}[theorem]{Lemma}

\crefname{claim}{Claim}{Claims}

\newtheorem{conjecture}[theorem]{Conjecture}
\newtheorem*{question*}{Question}

\theoremstyle{definition}

\newtheorem*{definition*}{Definition}

\theoremstyle{remark}
\newtheorem*{remark}{Remark}

\newlist{enumthm}{enumerate}{1}
\setlist[enumthm]{label=\textup{(\roman*)},ref=\thethm(\roman*)}
\Crefname{enumthmi}{Theorem}{Theorems}

\newcommand{\mb}{\mathbb}

\newcommand{\mbm}{\mathbbm}
\newcommand{\mc}{\mathcal}
\newcommand{\mf}{\mathfrak}

\newcommand{\on}{\operatorname}
\newcommand{\wh}{\widehat}
\newcommand{\wt}{\widetilde}

\let\originalleft\left
\let\originalright\right
\renewcommand{\left}{\mathopen{}\mathclose\bgroup\originalleft}
\renewcommand{\right}{\aftergroup\egroup\originalright}

\allowdisplaybreaks

\newif\ifpublic

\ifpublic

\newcommand{\ignore}[1]{}

\else

\fi

\title{Hitting time mixing for the random transposition walk}

\author[A1]{Vishesh Jain}
\address{Department of Mathematics, Statistics, and Computer Science, University of Illinois Chicago, Chicago, IL, 60607 USA}
\email{visheshj@uic.edu}

\author[A2]{Mehtaab Sawhney}
\address{Department of Mathematics, Columbia University, New York, NY 10027}
\email{m.sawhney@columbia.edu}

\begin{document}

\maketitle
\begin{abstract}
Consider shuffling a deck of $n$ cards, labeled $1$ through $n$, as follows: at each time step, pick one card uniformly with your right hand and another card, independently and uniformly with your left hand; then swap the cards. How long does it take until the deck is close to random? 

Diaconis and Shahshahani showed that this process undergoes cutoff in total variation distance at time $t = \lfloor n\log{n}/2 \rfloor$. Confirming a conjecture of N.~Berestycki, we prove the definitive ``hitting time'' version of this result: let $\tau$ denote the first time at which all cards have been touched. The total variation distance between the stopped distribution at $\tau$ and the uniform distribution on permutations is $o_n(1)$; this is best possible, since at time $\tau-1$, the total variation distance is at least $(1+o_n(1))e^{-1}$.   
\end{abstract}

\section{Introduction}\label{sec:introduction}
The object of study in this paper is the random transposition walk on the symmetric group $\mf{S}_n$. Informally, viewing elements of $\mf{S}_n$ as possible permutations of $n$ cards, this is the Markov chain with the following transitions: at each time step, pick one card uniformly with your right hand and another card, independently and uniformly with your left hand (note that with probability $1/n$, the cards will be the same); then swap the cards. From this point of view, perhaps the most natural question to ask is the following: what is the ``first time'' at which the cards are (approximately) shuffled?

Formally, let $P_n$ be the probability measure on $\mf{S}_n$ defined by 
\[P_n = \begin{cases}
\on{Id} & \text{ with probability }1/n\\
(ij) & \text{ with probability }2/n^2\text{ for }i<j.
\end{cases}\]
Then, the process described above can be modeled by the discrete time Markov chain given by $X_{t+1} = X_t\cdot \tau$ where $\tau \sim P_n$ (with $X_0 = \on{id}_n$, the identity permutation). In particular, $X_t \sim P_n^{\ast t}$, the $t$-fold convolution of $P_n$. 

Let $U_{\mf{S}_n}$ denote the uniform measure on $\mf{S}_n$. Since the random transposition walk is ergodic and reversible with respect to $U_{\mf{S}_n}$, the distribution of $X_t$ converges in total variation distance to $U_{\mf{S}_n}$ as $t \to \infty$; recall that for two (probability) measures $\mu$ and $\nu$ on a common ground set $E$, the total variation distance (or TV distance) and $L^1$-distance are defined by
\[d_{\on{TV}}(\mu,\nu) = \frac{d_{1}(\mu,\nu)}{2} = \frac{1}{2}\sum_{x\in E}|\mu(x) - \nu(x)|.\]

Returning to the question at the end of the first paragraph, a natural lower bound is the first time at which the random transposition walk has ``touched'' all the cards; formally, using the notation above, consider the random (stopping) time 
\[\tau := \min_{t\ge 1}\mbm{1}\bigg[\Big(\min_{k\in [n]}\sum_{\ell = 1}^{t}\mbm{1}[i_\ell = k \vee j_\ell = k]\Big)\ge 1\bigg]. \]
Observe that at time $\tau - 1$, the cards are certainly not well-shuffled, even in an approximate sense -- indeed, $X_{\tau-1}$ is supported entirely on permutations with at least one fixed point, whereas it is classical that a random permutation drawn from $\mf{S}_n$ has no fixed points with probability $(1+o_n(1))e^{-1}$. Moreover, since $X_{\tau}$ is not distributed exactly uniformly on $\mf{S}_n$ (already for $n=2$, a direct computation shows that $X_{\tau}(\on{id}_2) = 1/6$), the best scenario one can hope for is that $X_{\tau}$ is distributed asymptotically uniformly (in the limit $n \to \infty$); in fact, this has been conjectured by Nathana\"{e}l Berestycki (see \cite[Conjecture~1.2]{Tey20}). Our main result confirms this conjecture. 
\begin{theorem}\label{thm:main}
With notation as above, 
\[d_{\on{TV}}(X_{\tau},U_{\mf{S}_n})\leq \exp\big(-(\log n)^{1/2+o_n(1)}\big).\]
\end{theorem}

\subsection{Background and additional results} There is a rich and extensive literature surrounding the random transposition walk, which we only briefly summarize here; we refer the reader to \cite[Section~1.3]{BS19} for a more detailed account of the early history of this problem. In a pioneering work \cite{DS81} from 1981, Diaconis and Shahshahani \cite{DS81} used techniques from non-commutative Fourier analysis/representation theory to show that the random transposition walk undergoes cutoff at $n\log n/2$ with cutoff window $O(n)$,~i.e.
\[d_{\on{TV}}(P_n^{\ast \lfloor n\log{n}/2 + cn \rfloor}, U_{\mf{S}_n}) \xrightarrow[c \to \infty]{} 0 \quad \text{ and } \quad d_{\on{TV}}(P_n^{\ast \lfloor n\log{n}/2 - cn \rfloor}, U_{\mf{S}_n}) \xrightarrow[c \to -\infty]{n \to \infty} 1.\]
 To use terminology common in probabilistic combinatorics, whereas this result determines the ``sharp threshold'' for TV-mixing, \cref{thm:main} establishes the ``hitting time'' for TV-mixing; this is typically a much more challenging task since it involves a detailed understanding of the process in the threshold window. We note that the result of \cite{DS81} has been generalized from transpositions to $k$-cycles for fixed $k$ by Berestycki, Schramm and Zeitouni \cite{BSZ11}, and even further to conjugacy classes with support $o(n)$ by Berestycki and  {\c{S}}eng{\"u}l \cite{BS19}, both using probabilistic arguments. See also work of Saloff-Coste and Z\'uniga \cite{SZ08} for a refinement of \cite{DS81} for the $L^2$-cutoff. 
 
 Relevant to the present paper is a recent work of Teyssier \cite{Tey20} which established the cutoff profile for the random transposition walk: for fixed $c \in \mb{R}$,
\[d_{\on{TV}}(P_n^{\ast \lfloor n\log{n}/2 + cn \rfloor}, U_{\mf{S}_n}) \xrightarrow[n \to \infty]{} d_{\on{TV}}(\on{Pois}(1+e^{-2c}), \on{Pois}(1)),\]
where $\on{Pois}(a)$ denotes the Poisson distribution with mean $a$ the same cutoff profile appears for several other random walks (see,~e.g.,~\cite{NO22,N24}). As a key intermediate step in our work, we establish the following strengthening of Teyssier's result, which provides a distributional approximation for $X_t$ near cutoff (instead of only finding the distance to uniformity). In order to state this result, we need some notation. Let
\[t = \bigg\lfloor \frac{n\log n}{2}\bigg\rfloor + t' \]
with $t'\in \mb{Z}$. We define 
\[\gamma_t = e^{-2t'/n}.\]
Let $\nu_t$ denote the measure on $\mf{S}_n$ defined by the following sampling process: first, sample $M_t \in \{0,1,\dots,n\}$ according to the distribution $\mb{P}[M_t = x] = \mb{P}[\on{Pois}(\gamma_t) = x]/\mb{P}[\on{Pois}(\gamma_t) \le n]$, then sample a uniformly random subset $S_t$ of size $M_t$ in $[n]$ (the truncation $M_t\le n$ ensures that $S_t$ is well-defined), and finally, sample a uniformly random element of $\mf{S}_{[n]\setminus S_t}$ and view it as an element of $\mf{S}_n$ by fixing all of the elements in $S_t$. 
\begin{theorem}\label{thm:distrib}
With notation as above, if $|t'|\le n(\log\log n/4 - \log\log\log n)$, then 
\[d_{\on{TV}}(X_t,\nu_t)\leq n^{-1+o(1)}.\]
\end{theorem}
\begin{remark}
Since $\nu_t$ is uniformly distributed on the set of permutations with a given number of fixed points, it is straightforward to prove that 
\[d_{\on{TV}}(\nu_t,U_{\mf{S}_n}) = d_{\on{TV}}(\on{Pois}(1+\gamma_t),\on{Pois}(1)) + o_n(1);\]
combined with \cref{thm:distrib}, this recovers the main result of Teyssier \cite[Theorem~1.1]{Tey20}. Additionally, identifying the stronger statement in \cref{thm:distrib} leads to a significantly simpler proof of the main result of \cite{Tey20}; in particular, our proof of \cref{thm:distrib} uses only the $L^2$-estimates from the work of Diaconis-Shahshahani \cite{DS81} along with an application of Young's rule, thereby completely bypassing the $L^1$-theory and intricate combinatorial analysis in \cite{Tey20}; see \cref{subsec:outline} for further discussion. 
\end{remark}

We discuss two additional lines of work which are related to our result. First, recall that for an ergodic Markov chain with stationary distribution $\pi$, a stopping time $\kappa$ is said to be a stationary time if $X_{\kappa} \sim \pi$; $\kappa$ is said to be a strong stationary time if it is stationary and additionally, $X_{\kappa}$ is independent of $\kappa$. In 1985, Broder devised a strong stationary time for the random transposition walk concentrated around time $2n\log{n}$; an improved strong stationary time concentrated around $n\log{n}$ was given by Matthews \cite{M88} who also claimed a further improvement to $n\log{n}/2$; however, his proof contained a subtle error which was only noticed nearly 30 years later by White. Subsequently, White \cite{W19} constructed a rather intricate strong stationary time concentrated in a window of size $O(n)$ around time $n\log{n}/2$; in particular, this establishes the analogue of \cite{DS81} for the separation distance. The relationship of the present work to \cite{W19} is as follows: White showed that there exists some stopping time $\kappa$, concentrated around $n\log{n}/2$, for which $X_{\kappa}$ is exactly uniformly distributed; in contrast, \cref{thm:main} shows that the at time $\tau$, which is a lower bound for mixing as discussed above, $X_{\tau}$ is (TV-)approximately uniform. We remark that extending White's result to the random $k$-cycle walk is an open problem, even for $k=3$; on the other hand, our analysis can be adapted to prove an analogous result for any fixed $k$ (see work of Hough \cite{Hou16} and Nestoridi and Olesker-Taylor \cite{NO22}).  

Second, note that the random transposition walk is naturally associated to a random graph process, where we add the edge $\{i,j\}$ to the graph whenever the transposition $(ij)$ is chosen. Let $G_t$ denote the corresponding (random) graph at time $t$ and let $C_t$ denote the largest connected component of $G_t$; by standard results in random graphs, $C_t$ is macroscopic if $t\geq cn/2$ for $c > 1$). A well-known result of Schramm \cite{S05} shows that the distribution of the lengths of the largest cycles of $X_t|_{C_t}$ is close to the distribution of the lengths of the largest cycles for the uniform measure on $\mf{S}_{C_t}$; the natural conclusion of this line of work would be to establish the following conjecture due to  Nathana\"{e}l Berestycki (see \cite[Conjecture~1.3]{Tey20}).

\begin{conjecture} With notation as above, suppose $t \geq cn/2$ for $c > 1$. Then,
    \[d_{\on{TV}}(X_t|_{C_t}, U_{\mf{S}_{C_t}}) = o_n(1).\]
\end{conjecture}

Our final main result confirms this conjecture for $t$ near cutoff. 

\begin{theorem}
    \label{thm:giant}
With notation as above, if $|t'|\le n(\log\log n/4 - \log\log\log n)$, then 
\[d_{\on{TV}}(X_t|_{C_t}, U_{\mf{S}_{C_t}})\leq n^{-1+o(1)}.\]    
\end{theorem}

\subsection{Proof outline}\label{subsec:outline} Our proof of \cref{thm:main} can be roughly decomposed into three steps, which we briefly discuss. 

The first step is to prove \cref{thm:distrib}, which characterizes, up to a vanishingly small error, the distribution of the random transposition walk at time $t$. As mentioned in the introduction, this is a strengthening of work of Teyssier \cite{Tey20} on the cutoff profile for the random transposition walk. Interestingly, our proof is considerably simpler, essentially due to the reason that once the correct statement has been identified, the task amounts to establishing an upper bound of the form $o_n(1)$ (as opposed to \cite{Tey20}, where the (macroscopic) total variation distance to the uniform distribution is computed asymptotically exactly). Our proof of \cref{thm:distrib} proceeds along the same lines as that of Diaconis-Shahshahani \cite{DS81} in that we pass from the $L^1$ distance to the $L^2$ distance using Cauchy-Schwarz, and then use Plancharel's formula along with the same character estimates as \cite{DS81}. The only point of departure is that the uniform distribution on $\mf{S}_n$ is replaced by the distribution $\nu_t$; however, the Fourier coefficients of $\nu_t$ can be calculated exactly using a simple application of Young's rule. In contrast, the proof in \cite{Tey20} is based on approximating the $L^1$ distance using the Fourier inversion formula, as well as fairly intricate combinatorial analysis involving the Murnaghan–Nakayama rule. We believe that our method of distributional approximation can be applied to obtain cutoff profiles in many settings.

The second step is to bootstrap \cref{thm:distrib} to show that at times sufficiently close to $n\log{n}/2$, the number of untouched points is an ``approximate sufficient statistic'' for the random transposition walk (\cref{lem:uniform-non-fix}) in the sense that, conditioned on the untouched set being $M \subseteq [n]$ (which is assumed to be not too large), the distribution of the random transposition walk is close to the uniform distribution on permutations which leave each element of $M$ fixed. The proof is carried out entirely in physical space. The key idea is to exploit the following ``self-reducibility'' of the random transposition walk: the distribution of the walk at time $t$, conditioned on the untouched set \emph{containing} $M$, coincides exactly with the distribution of the random transposition walk at time $t$ run on the ground set $[n]\setminus M$. To see why this is useful, observe that the indicator of the event that the untouched set at time $t$ is exactly $M$ can be written as a signed combination of events of the form ``the untouched set at time $t$ contains $T$'' using the principle of inclusion-exclusion. Applying \cref{thm:distrib} to each of these summands and performing some explicit calculations using the form of the distribution $\nu_t$ completes the proof. 

Finally, given \cref{lem:uniform-non-fix}, \cref{thm:giant} follows immediately by using the law of total probability and classical results from the theory of random graphs. The deduction of \cref{thm:giant} from \cref{thm:main} is also relatively standard, and can be viewed as a blend of the classical hitting-time result for connectivity in random graph theory along with Broder's strong stationary time for the random transposition walk.

\subsection{Notation}\label{subsec:not}
We use standard asymptotic notation throughout, as follows. For functions $f=f(n)$ and $g=g(n)$, we write $f=O(g)$ or $f \lesssim g$ to mean that there is a constant $C$ such that $|f(n)|\le C|g(n)|$ for sufficiently large $n$. Similarly, we write $f=\Omega(g)$ or $f \gtrsim g$ to mean that there is a constant $c>0$ such that $f(n)\ge c|g(n)|$ for sufficiently large $n$. Finally, we write $f\asymp g$ or $f=\Theta(g)$ to mean that $f\lesssim g$ and $g\lesssim f$, and we write $f=o(g)$ or $g=\omega(f)$ to mean that $f(n)/g(n)\to0$ as $n\to\infty$. Subscripts on asymptotic notation indicate quantities that should be treated as constants.

\subsection{Acknowledgements}\label{subsec:ack}
MS thanks Evita Nestoridi and Dominik Schmid for useful conversations. We thank Nathana\"{e}l Berestycki and Lucas Teyssier for comments on the manuscript. VJ is supported by NSF CAREER award DMS-2237646. This research was conducted during the period MS served as a Clay Research Fellow. 

\subsection{Organization} The remainder of this paper is organized as follows. In \cref{sec:reduc}, we show how to deduce \cref{thm:main,thm:giant} from \cref{thm:distrib}. The main work in this section is to prove \cref{lem:uniform-non-fix}, which we undertake in \cref{subsec:suff-stat}, before completing the proofs in \cref{subsec:proof-finish}. Finally, we prove \cref{thm:distrib} in \cref{sec:distrib}.

\section{Reduction to distributional approximation at a fixed time}\label{sec:reduc}

\subsection{An approximate sufficient statistic}
\label{subsec:suff-stat}
Throughout this section, we will make use of the following notation: denote the (random) sequence of transpositions chosen by the random transposition walk up to time $t$ by $(i_1,j_1),\dots, (i_t, j_t)$. For a subset $T \subseteq [n]$, let $\mc{G}(T) = \mc{G}^t(T)$ denote the event that $\{i_1, j_1,\dots, i_t, j_t\} \cap T = \emptyset$; in words, the transposition walk does not ``touch'' $T$ up to (and including) time $t$. Let $\mc{F}^t(T) = \mc{F}(T)$ denote the event $\mc{G}(T) \cap \bigcap_{j \in T^c}\mc{G}^{c}(\{j\})$; in words, the transposition walk touches everything outside $T$ and does not touch $T$ up to (and including) time $t$. For convenience of notation, we will denote $\mc{F}([M])$ and $\mc{G}([M])$ by $\mc{F}(M)$ and $\mc{G}(M)$ for an integer $1\leq M \leq n$.

We begin with the following proposition, which is central to proving our main results. Roughly, it asserts that for times close to $n\log n/2$, the ``untouched set'' is an approximate sufficient statistic in the sense that the distribution of the random transposition walk, conditioned on the untouched set being $S$, is close to the uniform distribution on $\mf{S}_{[n]\setminus S}$ (provided that $S$ is not too large). More precisely, we have the following.  

\begin{proposition}\label{lem:uniform-non-fix}
Fix $t$ and $M$ such that $|t'|\le n(\log\log n/4 -2\log\log\log n)$ and $M\le (\log n)^2$. 
Let $X_t^{M}$ denote the random variable $X_t$ conditioned on the event $\mc{F}(M)$. Let $\mu^M$ denote the measure on $\mf{S}_n$ induced by the the uniform distribution on $\mf{S}_{n\setminus [M]}$ via the natural inclusion $\mf{S}_{n\setminus [M]} \hookrightarrow \mf{S}_n$. Then 
\[d_{\on{TV}}({X}_t^{M},\mu^M) \le n^{-1+o(1)}.\]
\end{proposition}

The proof of this proposition requires a few intermediate lemmas. The first lemma provides a relationship between the probabilities of the events $\mc{F}(T)$ and $\mc{G}(T)$. Recall that $\gamma_t = \exp(-2t'/n)$.  

\begin{lemma}
\label{lem:f-g}
Let $M \in \mb{N}$ and $T \subseteq [n]$ such that $M \leq T = n^{o(1)}$. Then,
\begin{itemize}
    \item $\mb{P}[\mc{G}(T)] = (1+n^{-1+o(1)}){\mb{P}[\mc{G}(M)]\gamma_t^{|T|-M}}{n^{M-|T|}}$,
    \item $\mb{P}[\mc{F}(M)] = (1+n^{-1+o(1)})\mb{P}[\mc{G}(M)]\exp(-\gamma_t)$, 
    \item $\sum_{T \supseteq [M], T \leq n^{o(1)}} \mb{P}[\mc{G}(T)] \leq n^{o(1)}\mb{P}[\mc{F}(M)]$. 
\end{itemize}
\end{lemma}
\begin{proof}
For the first item, note that
\begin{align*}
\label{eq:T-excl-bound}
    \mb{P}[\mc{G}(T)] = \left(\frac{n-|T|}{n}\right)^{2t} &= (1 + n^{-1+o(1)})\exp\left(-\frac{2t|T|}{n}\right) \nonumber \\
    &= (1+n^{-1+o(1)})\exp\left(-\frac{2tM}{n}\right)\exp\left(-\frac{2t(|T|-M)}{n}\right) \nonumber \\
    &= (1+n^{-1+o(1)})\frac{\mb{P}[\mc{G}(M)]\gamma_t^{|T|-M}}{n^{|T|-M}}.
\end{align*}
For the second item, the key point is that by the principle of inclusion-exclusion, \[\mbm{1}_{\mc{F}(M)} = \sum_{T\supseteq [M]}(-1)^{|T|-M}\mbm{1}_{\mc{G}(T)}.\]
Therefore, by the Bonferroni inequality, we have for any $k \in \mb{Z}^{\geq 0}$ that
\[\sum_{T\supseteq [M], |T| \leq M+2k+1}(-1)^{|T|-M}\mbm{1}_{\mc{G}(T)}\leq \mbm{1}_{\mc{F}(M)} \leq \sum_{T\supseteq [M], |T| \leq M+2k}(-1)^{|T|-M}\mbm{1}_{\mc{G}(T)}.\]
Taking expectations (with, say, $k = \lfloor (\log n)^2 \rfloor$) and using the first item along with $\gamma_t = o(\log n)$, we have
\begin{align*}
\label{eq:f-g-ratio}
    (1+n^{-1+o(1)})\frac{\mb{P}[\mc{F}(M)]}{\mb{P}[\mc{G}(M)]} &= \sum_{j=0}^{2k}\binom{n-M}{j} \frac{(-\gamma_t)^{j}}{n^j} \pm \binom{n-M}{2k+1}\frac{(-\gamma_t)^{2k+1}}{n^{k+1}} \nonumber \\
    &= (1+n^{-1+o(1)})\left(\sum_{j=0}^{2k}\frac{(-\gamma_t)^j}{j!} \pm \frac{(-\gamma_t)^{2k+1}}{(2k+1)!}\right) \nonumber \\
    &= (1+n^{-1+o(1)})\exp(-\gamma_t). 
\end{align*}
Finally, for the third item, we have by combining the first two items that
\begin{align*}
    \sum_{T\supseteq [M], T \leq n^{o(1)}} \mb{P}[\mc{G}(T)]
    &\lesssim \sum_{s=0}^{n^{o(1)}}\binom{n-M}{s}\mb{P}[\mc{F}(M)]e^{\gamma_t} \left(\frac{\gamma_t}{n}\right)^{s}\\
    &\lesssim \mb{P}[\mc{F}(M)]e^{\gamma_t}\sum_{s=0}^{n^{o(1)}}\left(\frac{3\gamma_t}{s}\right)^{s} \leq \mb{P}[\mc{F}(M)]\cdot n^{o(1)},
\end{align*}
where in the final inequality, we have used that $\gamma_t \leq \sqrt{\log n}$. \qedhere
\end{proof}

In order to state the next lemma, we need some notation. Let $t'\in \mb{R}$ be as in the statement of \cref{lem:uniform-non-fix}. For any $T\subseteq [n]$, let 
\[t_T = \lfloor 2^{-1}(n-|T|)\log(n-|T|) \rfloor + t'(n-|T|)/n.\]
We let $\nu^T = \nu^T_{t_T}$ denote the distribution on $\mf{S}_{[n]\setminus T}$ which is defined in the same way as the distribution $\nu_t$ appearing in the statement of \cref{thm:distrib} with two modifications: the time $t$ is replaced by $t_T$ and the ground set $[n]$ is replaced by $[n]\setminus T$. Using the natural inclusion $\mf{S}_{[n]\setminus T} \hookrightarrow \mf{S}_n$, we can (and will) view $\nu^T$ as a distribution on $\mf{S}_n$. Note that, by definition, 
\[\gamma_{t_T} = e^{-2t'(n-|T|)/n(n-|T|)} = e^{-2t'/n} = \gamma_t.\]

\begin{lemma}
    \label{lem:n-m-factorial}
For any $\sigma \in \mf{S}_n$ such that $\on{Fix}(\sigma) \supseteq [M]$ and $|\on{Fix}(\sigma)| \leq n^{o(1)}$, we have that
\[\sum_{[M] \subseteq T \subseteq \on{Fix}(\sigma)}(-1)^{|T|-M}\mb{P}[\mc{G}(T)]\cdot \nu^T(\sigma) = (1+n^{-1+o(1)})\frac{\mb{P}[\mc{F}(M)]}{(n-M)!}.\]
\end{lemma}
\begin{proof}
    Consider $\sigma \in \mf{S}_n$ satisfying the assumptions of the lemma and let $T\subseteq \on{Fix}(\sigma)$. By definition, 
    \begin{align*}
        \nu^T(\sigma) 
        &= \sum_{r= 0}^{n - |T|} \frac{\mb{P}[\on{Pois}(\gamma_{t_T}) = r]}{\mb{P}[\on{Pois}(\gamma_{t_T}) \leq n - |T|]}\cdot \frac{\binom{|\on{Fix(\sigma)}| - |T|}{r}}{\binom{n-|T|}{r}(n-|T|-r)!}.
    \end{align*}
    Since $\gamma_{t_T} = \gamma_t$ and for any $0 \leq m \leq s \leq n$ with $s = n^{o(1)}$ and $r = n^{o(1)}$, 
    \begin{align*}
    \frac{\mb{P}[\on{Pois}(\gamma_{t_T}) = r]}{n^{s-m}\cdot \binom{n-s}{r}(n-r-s)!} = \frac{e^{-\gamma_{t_T}}(\gamma_{t_T})^{r}}{r!}\cdot \frac{1}{n^{s-m}\cdot \binom{n}{r}(n-r-s)!}   
    = \frac{e^{-\gamma_{t_T}}(\gamma_{t_T})^{r}}{(n-m)!}(1+n^{-1 + o(1)}), 
\end{align*}
it follows that
\begin{align*}
\nu^T(\sigma) = (1+n^{-1+o(1)}) \frac{e^{-\gamma_t}n^{|T|-M}}{(n-M)!}\sum_{r=0}^{\on{|Fix(\sigma)|}-|T|}\gamma_t^{r}\binom{|\on{Fix(\sigma)}| - |T|}{r}.
\end{align*}
Therefore, by \cref{lem:f-g}, we have
\begin{align*}
    (1+n^{-1+o(1)})(n-M)!&\sum_{[M]\subseteq T \subseteq \on{Fix}(\sigma)}(-1)^{|T|-M}\mb{P}[\mc{G}(T)]\nu^T(\sigma)\\
    &= \mb{P}[\mc{F}(M)]\sum_{[M]\subseteq T \subseteq \on{Fix}(\sigma)}(-\gamma_t)^{|T|-M}\sum_{r=0}^{\on{|Fix(\sigma)|}-|T|}\gamma_t^{r}\binom{|\on{Fix(\sigma)}| - |T|}{r}\\
    &= \mb{P}[\mc{F}(M)]\sum_{[M]\subseteq T \subseteq \on{Fix}(\sigma)}(-\gamma_t)^{|T|-M}(1+\gamma_t)^{|\on{Fix(\sigma)}| - |T|}\\
    &= \mb{P}[\mc{F}(M)]\sum_{s\geq 0}\binom{\on{Fix(\sigma)}-M}{s}(1+\gamma_t)^{s}(-\gamma_t)^{\on{Fix(\sigma)}-M-s}\\
    &= \mb{P}[\mc{F}(M)]. \qedhere
\end{align*}
\end{proof}

The next lemma is the crucial juncture where we exploit the ``self-reducibility'' of the random transposition walk. 

\begin{lemma}
\label{lem:self-reducible}
Fix $t$ such that $|t'|\le n(\log\log n/4 -2\log\log\log n)$. For $T\subseteq [n]$, 
let $\wt{X}_t^{T}$ denote the random variable $X_t$ conditioned on the event $\mc{G}(T)$. Let $\nu^T$ denote the distribution on $\mf{S}_n$ appearing in the statement of \cref{lem:n-m-factorial}. If $T \leq n^{o(1)}$, then 
\[d_{\on{TV}}({\wt{X}}_t^{T},\nu^T) \le n^{-1+o(1)}.\]    
\end{lemma}

\begin{proof}
Let $\wt{\nu}^T$ be defined in a similar manner to $\nu^T$ with the time $t_T$ replaced by
\[\wt{t}_T = \lfloor 2^{-1}(n-|T|)\log(n-|T|) \rfloor + t'/n.\]
Since $|\gamma_{t_T} -\gamma_{\wt{t}_T}| = n^{-1+o(1)}$ and $\on{Pois}(a+b)$ has the same distribution as the independent sum $\on{Pois}(a) + \on{Pois}(b)$ for $a,b \geq 0$, it readily follows that $d_{\on{TV}}(\nu_T, \wt{\nu}_T) \leq n^{-1+o(1)}$. Therefore, it suffices to show that $d_{\on{TV}}(\wt{X}_t^T, \wt{\nu}^T) \leq n^{-1+o(1)}$. But this follows immediately from \cref{thm:distrib} upon noting that (i) $X_t \mid \mc{G}(T) \sim X_t^T$, the random transposition walk on $[n]\setminus T$ at time $t$ and (ii) the assumed bound on $|t'|$ and $|T|$ implies that
\[|t'| + T\log n \leq (n-|T|)(\log\log(n-|T|)/4 - \log\log\log(n-|T|))\]
so that \cref{thm:distrib} is indeed applicable for $[n]\setminus T$ at time $t$.  \qedhere
\end{proof}

We are now in position to prove \cref{lem:uniform-non-fix}. 

\begin{proof}[Proof of \cref{lem:uniform-non-fix}]

Let $\mf{S}_n^M \subseteq \mf{S}_n$ denote the subset of permutations which fix each element of $[M]$. The conclusion of \cref{lem:uniform-non-fix} asserts that
\begin{equation}
\label{eq:tv-explicit}
\sum_{\sigma \in \mf{S}_n^{M}}\Big|\mb{P}[\{X_{t} = \sigma\} \wedge \mc{F}(M)] - \mb{P}[\mc{F}(M)]\cdot (n-M)!^{-1}\Big| \le n^{-1+o(1)}\cdot \mb{P}[\mc{F}(M)].
\end{equation}
We break up the sum on the left hand side into two parts depending on the size of $\on{Fix}(\sigma)$. Namely, let $\mc{S} = \mf{S}_n^{M} \cap \{\sigma \in \mf{S}_n : |\on{Fix}(\sigma)| \leq (\log n)^3\}$ and let $\mc{L} = \mf{S}_n^M \setminus \mc{S}$. 

\paragraph{\bf Contribution from $\mc{L}$} Let $\mu^M$ be the distribution appearing in the statement of \cref{lem:uniform-non-fix} and let $\nu^M$ denote the distribution appearing in the statement of \cref{lem:n-m-factorial}. We have
\begin{align}
\label{eq:X-many-fixed-points}
    \sum_{\sigma \in \mc{L}}|\mb{P}[\wt{X}_t^M = \sigma]| &\leq n^{-1+o(1)} + \nu^M(\mc{L}) \nonumber \\ &\leq n^{-1+o(1)} + \mb{P}[\on{Pois}(\gamma_t) \geq (\log n)^2] + \mu^{M+(\log n)^2}(\mc{L}) \leq n^{-1+o(1)};
\end{align}
here, the first inequality uses \cref{lem:self-reducible} and the last inequality uses the basic combinatorial fact that the probability of a random permutation having exactly $k$ fixed points is at most $1/k!$ (see,~e.g.~[Eq~3.3]\cite{Tey20}).
Therefore,
\begin{align*}
    \sum_{\sigma \in \mc{L}}\mb{P}[\{X_t = \sigma\} \wedge \mc{F}(M)] \leq \sum_{\sigma \in \mc{L}}\mb{P}[\{X_t = \sigma\} \wedge \mc{G}(M)] \leq n^{-1 + o(1)}\mb{P}[\mc{G}(M)] \leq n^{-1 + o(1)}\mb{P}[\mc{F}(M)], 
\end{align*}
where the first inequality uses $\mc{F}(M)\subseteq \mc{G}(M)$, the second inequality uses \cref{eq:X-many-fixed-points}, and the final inequality uses the third item of \cref{lem:f-g}. 
Moreover, by the above combinatorial fact on fixed points of a random permutation, we have
\begin{align*}
    \sum_{\sigma \in \mc{L}}|\mb{P}[\mc{F}(M)](n-M)!^{-1}| = \mb{P}[\mc{F}(M)]\mu^M(\mc{L}) \leq n^{-1 + o(1)}\mb{P}[\mc{F}(M)].
\end{align*}
Combining the above two display equations with the triangle inequality shows that
\[\sum_{\sigma \in \mc{L}}\Big|\mb{P}[\{X_{t} = \sigma\} \wedge \mc{F}(M)] - \mb{P}[\mc{F}(M)]\cdot (n-M)!^{-1}\Big| \le n^{-1+o(1)}\cdot \mb{P}[\mc{F}(M)].\]

\paragraph{\bf Contribution from $\mc{S}$}

In this case, we have
 \begin{align*}
        \sum_{\sigma \in \mc{S}}&\Big|\mb{P}[\{X_{t} = \sigma\} \wedge \mc{F}(M)] - \mb{P}[\mc{F}(M)]\cdot (n-M)!^{-1}\Big| \\
        &\leq n^{-1+o(1)}\mb{P}[\mc{F}(M)] + \sum_{\sigma \in \mc{S}}\sum_{[M]\subseteq T \subseteq \on{Fix}(\sigma)}\Big|\mb{P}[\{X_{t} = \sigma\} \wedge \mc{G}(T)] - {\mb{P}[\mc{G}(T)]}\nu^T(\sigma)\Big| \\
        &\leq n^{-1+o(1)}\left(\mb{P}[\mc{F}(M)] + \sum_{[M]\subseteq T, |T| \leq n^{o(1)}}\mb{P}[\mc{G}(T)]\right)\\
        &\leq n^{-1+o(1)}\mb{P}[\mc{F}(M)].
 \end{align*}
 Here, for the first inequality, we have used the triangle inequality, \cref{lem:n-m-factorial}, and the following consequence of the principle of inclusion-exclusion:
 \[\mbm{1}_{X_t = \sigma}\cdot \mbm{1}_{\mc{F}(M)} = \sum_{[M]\subseteq T \subseteq \on{Fix}(\sigma)}(-1)^{|T|-M}\mbm{1}_{X_t=\sigma}\cdot \mbm{1}_{\mc{G}(T)};\]
 for the second inequality, we have used \cref{lem:self-reducible}, and for the final inequality, we have used the third item of \cref{lem:f-g}. 

 Together with the above estimate for $\sigma \in \mc{L}$, this completes the proof of \cref{eq:tv-explicit}. \qedhere
\end{proof}

%

\subsection{Finishing the proof}
\label{subsec:proof-finish}
Given the key \cref{lem:uniform-non-fix}, the deduction of \cref{thm:main,thm:giant} is relatively straightforward. 

\begin{proof}[Proof of \cref{thm:giant}] Let $t$ and $C_t$ be as in the statement of the theorem and let $S = S_t \subseteq [n]$ denote the set of indices which are not touched by time $t$.  Let $\mc{C}_t$ denote the event that $C_t = [n]\setminus S_t$ and $|S_t| \leq (\log n)^2$; by standard results in the theory of random graphs (see,~e.g.~\cite[Theorem~7.2]{B98}), it follows that $\mb{P}[\mc{C}_t] \leq n^{-1+o(1)}$. Therefore, it suffices to show that the assertion of \cref{thm:giant} holds, conditioned on the event $\mc{C}_t$; this follows immediately from \cref{lem:uniform-non-fix} using the law of total probability.  
\end{proof}

\begin{proof}[Proof of \cref{thm:main}]
Let $t^{\ast} = \Big\lfloor \frac{n(\log n)}{2}\Big\rfloor - \frac{n\log\log n}{4} + 2n\log\log\log n$. Let $S_{t^*}$ denote the set of indices which are not touched by time $t^*$. 
 
Consider the following two ``marking schemes'':

\begin{itemize} 
\item At time $t^*$, mark all of the elements in $[n]\setminus S_{t^*}$. At each subsequent step $s$, choose a random pair $(i_s j_s)$ according to the distribution $Q_n$ defined as follows:
\[Q_n = \begin{cases}
(i i) & \text{ with probability }2/n^2\text{ for }i \in S_{t^*}\\
(j j) & \text{with probability } n^{-2}\cdot (n-2|S_{t^*}|)/(n-|S_{t^*}|)\text{ for } j\notin S_{t^*}\\
(ij) & \text{ with probability }2/n^2\text{ for }i<j.
\end{cases}\]
Update the set of marked elements as follows: if the random transposition $(i_s, j_s)$ consists of an unmarked element and a marked element, or if it consists of the same unmarked element twice, then mark the unmarked element. Let $\kappa_m$ denote the first time when all of the elements are marked.  

\item Same as above, except the random pairs are chosen according to the distribution $P_n$ appearing in the definition of the random transposition walk. Let $\tau_m$ denote the first time when all of the elements are marked. 
\end{itemize}

Corresponding to the two marking schemes above, we consider the following two Markov processes:
\begin{itemize}
    \item $(Z_t)_{t\geq t^*}$, where $Z_{t^*}$ is sampled from the uniform distribution on $\mf{S}_{[n]\setminus S_{t^*}}$ (viewed as a distribution on $\mf{S}_n$) and the process evolves according to the random transposition walk with transitions sampled from $Q_n$.

    \item $(Y_t)_{t\geq t^*}$, defined as above, except with transitions sampled from $P_n$. 
\end{itemize}

The reason behind the introduction of $Q_n$ and $\kappa_m$ is the following: essentially due to Broder's argument, $Z_{\kappa_m} \sim U_{\mf{S}_n}$. To see this, let $M_t$ denote the set of marked elements at time $t$. We will prove by induction on $t\geq t^*$ that the distribution of $Z_t|_{M_t}$ is uniform over all of the $|M_t|!$ permutations. This is true at time $t^*$ by construction. Suppose it is true at time $t\geq t^*$. At time $t+1$, there are two possibilities. The first possibility is that $M_{t} = M_{t+1}$, so the transposition chosen was either contained entirely inside the marked set or is between two distinct unmarked elements; in either case, $X_{t+1}|_{M_{t+1}}$ is uniform due to the inductive hypothesis. The second possibility is that $M_{t+1} = M_{t} \cup \{v\}$ for some $v \notin M_{t}$. Since $X_t|_{M_t}$ is uniform, and since, by definition of $Q_n$, all of the transpositions $(v j)$, $j \in [n]$ are equally likely, it follows that $X_{t+1}|_{M_{t+1}}$ is also uniformly distributed. 

Since $Q_n$ and $P_n$ agree outside of the weight of the identity element, it follows by direct computation that $d_{\on{TV}}(P_n, Q_n) \leq |S_{t^*}|/n^2$. Moreover, it follows by direct computations as in \cref{lem:f-g} that 
\[\mb{P}[|S_{t^*}| \geq (\log n)^2] \leq n^{-\omega(1)}, \text{ and } \mb{P}[\max\{\tau_m, \kappa_m\} - t^* \geq n\log{n}] \leq n^{-\omega(1)}.\] 
In particular, except with probability $n^{-\omega(1)}$, the total variation distance between $P_n$ and $Q_n$ is $n^{-2+o(1)}$; since we need to run the process for only $n^{1+o(1)}$ steps (again, except with probability $n^{-\omega(1)}$), it follows that $\mb{P}[\tau_m \neq \kappa_m] \leq n^{-1+o(1)}$. Since $Z_{\kappa_m} \sim U_{\mf{S}_n}$, it therefore follows that 
\[d_{\on{TV}}(Y_{\tau_m}, U_{\mf{S}_n}) \leq n^{-1+o(1)}.\] 

By the definition of $t^*$, it follows from direct computation that $\mb{P}[\tau \leq t^*] \leq \exp(-(\log n)^{1/2 + o(1)})$. Moreover, on the event that $|S_{t^*}| = n^{o(1)}$, $\tau \leq n^{1+o(1)}$, and $\tau > t^*$, we have that $\tau = \tau_m$, except with probability $n^{-1+o(1)}$; this is because a transposition involving two elements of $S_{t^*}$ is chosen throughout the process with probability at most $|S_{t^*}|^2 n^{-1+o(1)} = n^{-1+o(1)}$. Therefore, as above, it follows that $\mb{P}[\tau = \tau_m] \geq 1 - \exp(-(\log n)^{1/2 + o(1)})$, so that it suffices to prove \cref{thm:main} with $X_\tau$ replaced by $X_{\tau_m}$.

By the above discussion, in order to prove the statement of the theorem, it suffices to show that $\on{d}_{TV}(X_{\tau_m}, Y_{\tau_m}) \leq n^{-1+o(1)}$. By the proof of \cref{thm:giant}, we have that $\on{d}_{TV}(X_{t^*}|_{[n]\setminus S_t^*}, Y_{t^*}) \leq n^{-1+o(1)}$. In particular, by the coupling characterization of total variation distance, there exists a coupling $(W,V)$ of these distributions such that $\mb{P}[W\neq V] \leq n^{-1+o(1)}$. Evolving this coupling using the same updates for the random transposition walk gives a coupling $(W',V')$ of $(X_{\tau_m}, Y_{\tau_m})$ such that $\mb{P}[W'\neq V'] \leq \mb{P}[W\neq V] \leq n^{-1+o(1)}$. Therefore, $d_{\on{TV}}(X_{\tau_m}, Y_{\tau_m}) \leq n^{-1+o(1)}$, as required. 
\end{proof}

\section{Approximating the distribution at a fixed time}\label{sec:distrib}

\subsection{Nonabelian Fourier Transform}\label{sec:non-abel}
In this subsection, we recall various standard notions concerning the (nonabelian) Fourier transform over finite groups. Given a finite group $G$, we let $\wh{G}$ denote the set of irreducible representations. 
We define the convolution of $f_1,f_2:G\to \mb{C}$ to be 
\[(f_1\ast f_2)(z) := \sum_{x\in G}f_1(x)f_2(x^{-1}z).\]
For $\rho \in \wh{G}$, let $d_{\rho}$ denote the dimension of the representation. The nonabelian Fourier transform for a function $f:G\to \mb{C}$ is the map $\wh{f} : \wh{G} \to \oplus_{\rho \in \wh{G}}(\mb{C}^{d_{\rho}\times d_{\rho}})$ given by
\[\wh{f}(\rho) = \sum_{g\in G}f(g)\rho(g).\]

Recall that for any $f_1, f_2 : G \to \mb{C}$ and $\rho \in \wh{G}$,
\[\wh{f_1\ast f_2}(\rho) = \wh{f_1}(\rho)\cdot \wh{f_2}(\rho).\]
In our applications, we will restrict attention to functions $f$ which are class functions (i.e.~are constant on conjugacy classes). For such functions, it follows from Schur's lemma (see e.g. \cite[Lemma~5]{DS81}) that 
\begin{equation}\label{eq:schur}
\wh{f}(\rho) = \bigg(\frac{\sum_{g\in G}f(g)\on{Tr}(\rho(g))}{d_{\rho}}\bigg)\cdot \on{Id}_{d_{\rho}}.
\end{equation}
Recall that the character of $\rho \in \wh{G}$ is defined by
\[\chi_{\rho}(g) = \on{Tr}(\rho(g)).\]
Observe that $\chi_{\rho}(g)$ is a class function (i.e.~constant on conjugacy classes). 

\subsection{Character estimates for $P_n^{\ast t}$}
The following is immediate from the notions introduced above and the definition of the random transposition walk.
\begin{lemma}\label{thm:charac-estim}
Let $\rho\in \wh{\mf{S}_n}$ and $\tau \in \mf{S}_n$ be an (arbitrary) transposition. Then
\[\wh{P_n}(\rho) := \bigg(\frac{1}{n} + \frac{n-1}{n} \cdot \frac{\chi_\rho(\tau)}{d_{\rho}}\bigg)\cdot \on{Id}_{d_{\rho}}.\] 
\end{lemma}

We need the following bound on the character ratio $\chi_{\rho}(\tau)/d_{\tau}$ appearing in the above expression. As is standard, we will identify elements of $\wh{\mf{S}_n}$ with (positive integer) partitions $\lambda$ of $n$. We use $\lambda \vdash n$ to denote that $\lambda$ is a partition of $n$. Given a partition $\lambda = (\lambda_1,\dots,\lambda_k)$ with $\lambda_1\geq \dots \geq \lambda_k$, we let $\lambda'$ denote the conjugate partition and $\lambda^* = (\lambda_2,\dots,\lambda_k)$ to be the partition of $n-\lambda_1$ obtained by truncating $\lambda$.  

\begin{lemma}\label{lem:char-eval}
Let $\lambda \vdash n$ and $\tau \in \mf{S}_n$ be a transposition. Then 
\[\chi_{\lambda}(\tau) = - \chi_{\lambda'}(\tau) = d_{\lambda} \binom{n}{2}^{-1}\bigg(\sum_{i=1}^{n}\binom{\lambda_i}{2} - \binom{\lambda_i'}{2}\bigg).\]
In particular, if $\lambda_1 \geq n - (\log n)^{2}$, then 
    \[\frac{\chi_{\lambda}(\tau)}{d_{\lambda}} = \frac{\binom{\lambda_1}{2}}{\binom{n}{2}} + O(n^{-2+o(1)}) = 1-\frac{2(n-\lambda_1)}{n}+ O(n^{-2+o(1)}).\]

\end{lemma}
The evaluation $\chi_{\lambda}/d_{\lambda}(\tau)$ when $\tau$ is a transposition appears as \cite[Lemma~7]{DS81}. The bound for $\lambda_1 \geq n-(\log n)^2$ follows by direct computation using the evaluation. 

We also need the following bounds on the dimension of irreducible representations which follow from the hook-length formula. 

\begin{lemma}\label{prop:hook-prop}
Let $\lambda\vdash n$. Then,
\begin{itemize}
\item \[d_{\lambda} \leq \binom{n}{\lambda_1}\cdot \sqrt{(n-\lambda_1)!};\]

\item for $\lambda_1 = n-x$ with $x \leq (\log n)^2$,
\[d_\lambda = \binom{n}{\lambda}\cdot d_{\lambda^{\ast}}\left(1 - \frac{x}{n} + O(n^{-2 + o(1)}) \right).\]
\end{itemize}

\end{lemma}
The first item is \cite[Corollary~2]{DS81}. The second item is a mild extension of \cite[Proposition~3.2]{Tey20} and follows from exactly the same proof there.  

Given the character estimates in \cref{lem:char-eval}, it is natural to decompose the irreducible representations based on the size of the largest part. Accordingly, let
\begin{align*}
\mc{L} &= \{\lambda_1\vdash n: \lambda_1\ge n - (\log n)^2\}.
\end{align*}

The next lemma allows us to control the mass of the Fourier coefficients of $P_n^{\ast t}$ on representations outside $\mc{L}$. The proof essentially appears in \cite[pg.~168-174]{DS81} (in \cite{DS81}, the corresponding result is stated only for $t'\geq 0$; an extension to negative $t'$ for the $L^1$ case is provided in \cite[Section~4.1]{Tey20}; for our purpose, owing to the bound on $t'$ and the definition of $\mc{L}$, we are able to directly use the estimates in \cite{DS81}). For completeness, we record the details in \cref{sec:deferred-proofs}.

\begin{lemma}\label{lem:removal}
Let $t$ be such that $|t'|\le n (\log \log n/4 - \log\log\log n)$. We have that 
\[\sum_{\substack{\lambda\vdash n\\\lambda\notin \mc{L}}} d_{\lambda}^2\cdot \bigg|\frac{1}{n} + \frac{n-1}{n} \cdot \frac{\chi_{\lambda}(\tau)}{d_{\lambda}}\bigg|^{2t}\leq n^{- 2+o(1)}.\]
\end{lemma}

\subsection{Character estimates for $\nu_t$}

In order to estimate the Fourier coefficients of $\nu_t$, we define $\xi_M$ to be the distribution on $\mf{S}_n$ obtained as follows: sample a uniform subset $S\subseteq [n]$ of size $M$, then sample a uniformly random permutation of $\mf{S}_{[n]\setminus S}$, and extend this to an element of $\mf{S}_n$ by fixing each element in $S$. It turns out that $\wh{\xi_M}$ vanishes on all irreducible representations $\lambda$ with $\lambda_1 < n-M$; on the remaining irreducible representations, we have an explicit evaluation of $\wh{\xi_M}$ (provided that $M$ is not too large). More precisely, we have the following.  

\begin{lemma}\label{lem:key-comp}
Let $\lambda\vdash n$.
\begin{itemize}
    \item If $\lambda_1 < n-M$, then $\wh{\xi_M}(\lambda) = 0$. 

    \item If $\lambda_1 \geq n-M$ and $M \leq n/3$, then
    \[\wh{\xi_M}(\lambda) = \frac{d_{\lambda^{\ast}}\binom{M}{n-\lambda_1}}{d_{\lambda}} \on{Id}_{d_{\lambda}}.\]
\end{itemize}
\end{lemma}
The proof of this lemma utilizes several basic facts about representations of the symmetric group; we refer the reader to \cite{Sag13} for a very readable account. 
\begin{proof}
    Since $\xi = \xi_M$ is a class function (i.e.~conjugation invariant), it follows from Schur's lemma that for any irreducible representation $\lambda$, $\wh{\xi}(\lambda) = a_{\lambda}\on{Id}_{d_\lambda}$, where
    \[a_{\lambda} = \frac{\sum_{\sigma \in \mf{S}_n}\xi(\sigma)\chi^{\lambda}(\sigma)}{d_{\lambda}}.\]
Consider the partition $\mu = \mu_M = (n-M,1,\dots,1)$ of $n$; the corresponding Young subgroup of $\mf{S}_n$ \cite[Definition~2.1.2]{Sag13} is given by
\[\mf{S}_{\mu} = \mf{S}_{\{1\}} \times \cdots \times \mf{S}_{\{M\}} \times \mf{S}_{\{M+1,\dots,n\}}.\]
Let $\psi = \psi^{\on{triv}}$ denote the trivial character on $\mf{S}_{\mu}$ and let $\psi^{\uparrow \mf{S}_n}$ denote the character of the induced representation \cite[Definition~1.12.2]{Sag13} on $\mf{S}_n$. Also, let $\chi^{\lambda}_{\downarrow \mf{S}_{\mu}}$ denote the character of the restricted reprsentation on $\mf{S}_{\mu}$. 
Then, by definition of $\xi$, we have that
\begin{align*}
    \sum_{\sigma \in \mf{S}_n}\xi(\sigma)\chi^{\lambda}(\sigma) &= \binom{n}{M}\sum_{\sigma \in \mf{S}_{\mu}}\frac{\chi^{\lambda}_{\downarrow \mf{S}_{\mu}}(\sigma)}{\binom{n}{M}(n-M)!}\\
    &= \frac{1}{(n-M)!}\sum_{\sigma \in \mf{S}_{\mu}}\psi(\sigma)\chi^{\lambda}_{\downarrow \mf{S}_{\mu}}(\sigma)\\
    &= \frac{1}{n!}\sum_{\sigma \in \mf{S}_n}\psi^{\uparrow \mf{S}_n}(\sigma) \chi^{\lambda}(\sigma),
\end{align*}
where the last equality is Frobenius reciprocity \cite[Theorem~1.12.6]{Sag13}. 
Since $\psi^{\uparrow \mf{S}_n}$ is the character for the permutation representation associated to $\mu$ \cite[Definition~2.1.5]{Sag13}, it follows from Young's rule \cite[Theorem~2.11.2]{Sag13} that
\[\psi^{\uparrow \mf{S}_n} = \sum_{\alpha \vdash n} K_{\alpha \mu} \chi^{\alpha}(\sigma),\]
where $K_{\alpha \mu}$ are the Kostka numbers. Hence, by the orthonormality of characters, it follows that
\[
\sum_{\sigma \in \mf{S}_n}\xi(\sigma)\chi^{\lambda}(\sigma) = K_{\lambda \mu}.
\]
Finally, let us evaluate the Kostka numbers $K_{\lambda \mu}$. By definition \cite[Definition~2.11.1]{Sag13}, these are equal to the number of semistandard Young tableaux \cite[Definition~2.9.5]{Sag13} with $n-M$ occurrences of $1$ and a single occurrence each of $2,\dots, M+1$. 
\begin{itemize}
    \item If $\lambda_1 < n-M$, then $K_{\lambda \mu} = 0$, since all the occurences of $1$ must necessarily be in the first row. This proves the first item. 

    \item Now, suppose $\lambda_1 \geq n-M$ and $M \leq n/3$. Notice that any valid semistandard Young tableaux can be obtained as follows: (i) first, we put $1$ in the leftmost $n-M$ entries of the top row (ii) next, we choose $n-\lambda_1$ elements of $\{2,\dots,M+1\}$ and populate rows $2$ and onwards with a standard Young tableaux of shape $\lambda^*$ with these entries \cite[Definition~2.5.1]{Sag13} (iii) finally, we place the remaining entries in increasing order in the top row. Therefore,
    \begin{align*}
        K_{\lambda \mu} &= \binom{M}{n-\lambda_1}\cdot \#\text{ of standard Young tableaux of shape $\lambda^*$}\\
        &= \binom{M}{n-\lambda_1}\cdot d_{\lambda^*}, 
    \end{align*}
    where the last equality follows from \cite[Theorem~2.6.5]{Sag13}. This completes the proof of the second item. \qedhere
\end{itemize}
\end{proof}

\subsection{Proof of \texorpdfstring{\cref{thm:distrib}}{}} We have the necessary preparation to prove \cref{thm:distrib}.
\begin{proof}[{Proof of \cref{thm:distrib}}]
Recall that $t = \lfloor n\log{n}/2 \rfloor + t'$, $|t'|\leq n(\log \log n/4 - 2\log\log\log n)$ and $\gamma_t = e^{-2t'/n}$. In order to simplify computations, we consider the following truncated version $\mu_t$ of $\nu_t$, which is defined in an identical fashion, except that the size $M_t$ of the random subset $S_t$ has the distribution $\mb{P}[M_t = x] = \mb{P}[\on{Pois}(\gamma_t) = x]/\mb{P}[\on{Pois}(\gamma_t) \leq (\log n)^2]$, for $x \leq (\log n)^2$.

Since $\gamma_t = o(\log n)$ by our assumption on $|t'|$, it follows via standard bounds on Poisson random variables that 
\[\mb{P}[\on{Pois}(\gamma_t)\ge (\log n)^2] \le (\log n)^{-\Omega((\log n)^2)} = n^{-\omega(1)},\]
from which, it readily follows using the natural coupling that $d_{\on{TV}}(\nu_t, \mu_t) = n^{-\omega(1)}$. Thus, in order to prove \cref{thm:distrib}, it suffices to prove that 
\[d_{\on{TV}}(X_t,\mu_t)\leq n^{-1+o(1)}.\]
Let $f_1$ be the probability density function of $X_t$ and $f_2$ be the probability density function of $\mu_t$. Using Cauchy-Schwarz followed by the Plancharel formula, as in \cite{DS81}, we have that
\begin{align*}
    4\cdot d_{\on{TV}}(X_t, \mu_t)^2 &\leq n!\sum_{\sigma \in \mf{S}_n}(f_1(\sigma) - f_2(\sigma))^2 \\
    &= \sum_{\lambda \in \wh{\mf{S}_n}}d_{\lambda}\on{Tr}\left((\wh{f_1}(\lambda) - \wh{f_2}(\lambda))(\wh{f_1}(\lambda) - \wh{f_2}(\lambda))^{\dagger}\right)\\
    &= \sum_{\lambda \notin \mc{L}}d_{\lambda}^2\bigg|\frac{1}{n} + \frac{n-1}{n} \cdot \frac{\chi_{\lambda}(\tau)}{d_{\lambda}}\bigg|^{2t} 
   + \sum_{\lambda \in \mc{L}}d_{\lambda}\on{Tr}\left((\wh{f_1}(\lambda) - \wh{f_2}(\lambda))(\wh{f_1}(\lambda) - \wh{f_2}(\lambda))^{\dagger}\right)\\
   &\leq n^{-2+o(1)} + \sum_{\lambda \in \mc{L}}d_{\lambda}\on{Tr}\left((\wh{f_1}(\lambda) - \wh{f_2}(\lambda))(\wh{f_1}(\lambda) - \wh{f_2}(\lambda))^{\dagger}\right);
\end{align*}
here, the third line follows from \cref{thm:charac-estim,lem:key-comp} and the final line follows from \cref{lem:removal}.

We are now left with bounding  
\begin{equation}
\label{eq:sum-L}
\sum_{\lambda \in \mc{L}}d_{\lambda}\on{Tr}\left((\wh{f_1}(\lambda) - \wh{f_2}(\lambda))(\wh{f_1}(\lambda) - \wh{f_2}(\lambda))^{\dagger}\right).
\end{equation}
By definition of $\mc{L}$, $\lambda_1 = n-x$ satisfies $x \leq (\log{n})^2$. Therefore, by the third item in \cref{lem:char-eval}, we have that 
\begin{align*}
    \wh{f_1}(\lambda) &= \bigg(\frac{1}{n} + \frac{(n-1)}{n}\cdot \frac{\chi_{\lambda}(\tau)}{d_{\lambda}}\bigg)^{t}\cdot \on{Id}_{d_{\lambda}}\\
    &= \bigg(1 - \frac{2x}{n}\cdot \frac{n-1}{n} + O(n^{-2 + o(1))}\bigg)^{t}\cdot \on{Id}_{d_{\lambda}}\\
    &= \exp\Big(\frac{-2xt}{n}\Big)\cdot (1 + O(t n^{-2+o(1)}))\cdot \on{Id}_{d_{\lambda}}.
\end{align*}

Furthermore, since
\[\mu_t \sim \sum_{\ell = 0}^{(\log n)^2} \xi_{\ell}\cdot \frac{\mb{P}[\on{Pois}(\gamma_t) = \ell]}{\mb{P}[\on{Pois}(\gamma_t) \leq (\log n)^2]},\]
where $\xi_{\ell}$ is the distribution in the statement of \cref{lem:key-comp}, 
it follows from \cref{lem:key-comp} and the second item in \cref{prop:hook-prop} that
\begin{align*}
\wh{f_2}(\lambda) &= \mb{P}[\on{Pois}(\gamma_t)\le (\log n)^2]^{-1}\cdot \sum_{\ell = 0}^{(\log n)^2}\frac{d_{\lambda^{\ast}}\binom{\ell}{x}}{d_{\lambda}}\cdot \mb{P}[\on{Pois}(\gamma_t) = \ell]\on{Id}_{d_{\lambda}}\\
&=(1+n^{-1+o(1)})\sum_{\ell = 0}^{(\log n)^2}\binom{\ell}{x} \cdot x! n^{-x} \cdot \frac{\gamma_t^{\ell}e^{-\gamma_t}}{\ell!}\on{Id}_{d_{\lambda}}\\
&= (1+n^{-1+o(1)})\sum_{\ell = x}^{(\log n)^2} n^{-x} \cdot \frac{\gamma_t^{\ell}e^{-\gamma_t}}{(\ell-x)!}\on{Id}_{d_{\lambda}}\\
&=(1+n^{-1+o(1)})n^{-x}\cdot \gamma_t^{x}\cdot \mb{P}[\on{Pois}(\gamma_t)\le (\log n)^2 - x]\on{Id}_{d_{\lambda}}\\
&= (1+n^{-1+o(1)})\exp\Big(\frac{-2xt}{n}\Big)\cdot \mb{P}[\on{Pois}(\gamma_t)\le (\log n)^2 - x]\cdot \on{Id}_{d_{\lambda}}.
\end{align*}

At this point, we are almost done. We break the sum in \cref{eq:sum-L} into two parts, depending on whether $x\ge (\log n)^2/2$ or not. 

\paragraph{\textbf{Case 1}: $x \geq (\log n)^2/2$} In this case, we use the crude bound $\wh{f_1}(\lambda) - \wh{f_2}(\lambda) = \alpha_{\lambda}\on{Id}_{d_{\lambda}}$ for $|\alpha_{\lambda}| \leq 3\exp(-2tx/n)$, along with the dimension bound $d_{\lambda} \leq n^{x}/\sqrt{x!}$ (\cref{prop:hook-prop}) to see that 
\begin{align*}
    \sum_{\lambda \in \mc{L}, x\geq (\log n)^2/2}d_{\lambda}\on{Tr}\left((\wh{f_1}(\lambda) - \wh{f_2}(\lambda))^2\right)
&\leq \sum_{x\in \mc{L}, x\geq (\log n)^2/2}d_{\lambda}^2|\alpha_{\lambda}|^2\\
&\leq 9\sum_{x = (\log n)^2/2}^{(\log n)^2}\sum_{\lambda \in \mc{L}, \lambda_1 = n-x}\exp(-4tx/n)\frac{n^{2x}}{x!}\\
&\leq  n^{o(1)} \max_{x \in [(\log n)^2/2, (\log n)^2]}\frac{\exp(-4t'x/n)}{x!} \leq n^{-\omega(1)},
\end{align*}
where in the last inequality, we used that $|t'|\leq n\log\log n/4$. 

\paragraph{\textbf{Case 2}: $x \leq (\log n)^2/2$} In this case, we have that $\mb{P}[\on{Pois}(\gamma_t)\le (\log n)^2 - x] = 1+n^{-\omega(1)}$. Therefore, $\wh{f_1}(\lambda) - \wh{f_2}(\lambda) = \alpha_\lambda \on{Id}_{d_\lambda}$ for $|\alpha_\lambda| \leq \exp(-2tx/n)\cdot n^{-1+o(1)}$. Using the same bounds on the dimension as above, we see that 
\begin{align*}
    \sum_{\lambda \in \mc{L}, x < (\log n)^2/2}d_{\lambda}\on{Tr}\left((\wh{f_1}(\lambda) - \wh{f_2}(\lambda))^2\right)
&\leq \sum_{x\in \mc{L}, x < (\log n)^2/2}d_{\lambda}^2|\alpha_{\lambda}|^2\\
&\leq n^{-2+o(1)}\sum_{x = 0}^{(\log n)^2/2}\sum_{\lambda \in \mc{L}, \lambda_1 = n-x}\exp(-4tx/n)\frac{n^{2x}}{x!}\\
&\leq  n^{-2+o(1)} \max_{x \in [0,(\log n)^2/2]}\frac{\exp(-4t'x/n)}{x!} \leq n^{-2+o(1)},
\end{align*}

For the last inequality, we used that when $x = o(\log n/\log \log n)$, then $\exp(4|t'|x/n) = n^{o(1)}$ and in the complementary range,
\[4|t'|x/n +2x \leq x\log x\]
using our bound on $|t'|$. \qedhere

\end{proof}

\subsection{Proof of \cref{lem:removal}}
\label{sec:deferred-proofs}
For completeness, we deduce \cref{lem:removal} from the estimates in \cite{DS81}.

\begin{proof}[Proof of \cref{lem:removal}]
Throughout, let $p(n)$ denote the partition number of $n$; we have via the Hardy--Ramanujan formula for partitions that $p(n)\lesssim e^{\pi\sqrt{2n/3}}$. Following \cite{DS81}, we break the sum appearing on the left hand side in the statement of \cref{lem:removal} into 5 parts: $S_1$ is the sum over $\lambda \vdash n$ satisfying $\lambda_1 \leq n/3$ and $\lambda_1' \leq n/3$; $S_2$ is the sum over $\lambda \vdash n$ satisfying ($n/3 < \lambda_1 \leq n/2$ and $\lambda_1' \leq n/2$) or ($n/3 < \lambda_1' \leq n/2$ and $\lambda_1 \leq n/2$); $S_3$ is the sum over $\lambda \vdash n$ satisfying $n/2 < \lambda_1 \leq 0.7n$ or $n/2 < \lambda_1' \leq 0.7n$; $S_4$ is the sum over $\lambda \vdash n$ satisfying $0.7n < \lambda_1 \leq n-(\log n)^2$, and finally, $S_5$ is the sum over $\lambda \vdash n$ satisfying $0.7n < \lambda_1' \leq n$. 

From \cite[Equation~3.2]{DS81}, we have $S_1 \le (1/3)^{2t}n! \leq n^{-\omega(1)}$; from \cite[Equation~3.3]{DS81}, we have $S_2 \lesssim 8^{n}(1/2)^{2t}n^{2n/3} \leq n^{-\omega(1)}$, and from \cite[Equation~3.5]{DS81}, we have $S_3 \lesssim 8^{n}(n/2)!(0.6)^{2t} \leq n^{-\omega(1)}$, where the final inequality uses the numerical inequality $\sqrt{e}\cdot 0.6 \leq 0.99$. 
 
 For the remaining summands, we have 
  \begin{align*}
        S_4 &\lesssim  \sum_{j=(\log n)^2}^{0.3n}\frac{p(j)}{j!}\exp\left(2j(j-1) \log{n}/n - 4t'(j/n -j^2/n^2 + j/n^2) \right)\\
        &\lesssim \sum_{j = (\log n)^2}^{0.3n}\frac{p(j)}{j!}\exp(2j^2\log{n}/n + 4|t'|j/n)\\
        &\lesssim \sum_{j = (\log n)^2}^{0.3n}\exp\left(j + 2j^2 \log{n}/n + j\log\log{n} - j\log{j}\right)
        \lesssim n^{-\omega(1)};
    \end{align*} 
    here, the first line follows from the proof of \cite[Equation~3.13]{DS81} (see the bottom of \cite[Page~172]{DS81}), and for the final estimate, we have used the easily verified numerical inequality: for $(\log n)^2 \leq j \leq 0.3n$ and all sufficiently large $n$
\[j + 2j^2\log{n}/n + j\log\log{n} \leq j\log{j} - j.\]
Finally, we have
 \begin{align*}
      S_5 &\lesssim e^{-4t/n}\sum_{j = 0}^{0.3n}\frac{p(j)}{j!}\exp\left(2j\log{n} - 4t(j/n - j^2/n^2)\right)\\
      &\lesssim e^{-4t/n}\sum_{j=0}^{0.3n}\frac{p(j)}{j!}\exp\left({2j^2\log{n}}/{n}\right)\exp\left(4|t'|j/n\right)\\
      &\lesssim e^{-4t/n}\sum_{j=0}^{0.3n}\exp(2j + 2j^2\log n/n - j\log j)\exp(j\log \log n)\exp(-4j\log\log\log{n})\\
      &\lesssim n^{-2 + o(1)};
 \end{align*}
 here, the first line follows from \cite[Equation~3.6]{DS81} and the first item of \cref{prop:hook-prop}, and in the final line, we have used the estimate $\exp(2j + 2j^2\log{n}/n + j\log\log n) = n^{o(1)}$ in the range $j \leq \log n/(\log \log n)^2$ and the estimate
 \[2j + 2j^2\log{n}/n + j\log\log n - j\log{j} - 4j\log\log\log{n} \leq -j\]
 in the complementary range. \qedhere
 
\end{proof}

\bibliographystyle{amsplain0.bst}
\bibliography{main.bib}

\providecommand{\bysame}{\leavevmode\hbox to3em{\hrulefill}\thinspace}
\providecommand{\MR}{\relax\ifhmode\unskip\space\fi MR }
\providecommand{\MRhref}[2]{%
  \href{http://www.ams.org/mathscinet-getitem?mr=#1}{#2}
}
\providecommand{\href}[2]{#2}
\begin{thebibliography}{10}

\bibitem{BSZ11}
Nathana{\"e}l Berestycki, Oded Schramm, and Ofer Zeitouni, \emph{Mixing times
  for random k-cycles and coalescence-fragmentation chains}, The Annals of
  Probability (2011), 1815--1843.

\bibitem{BS19}
Nathana{\"e}l Berestycki and Bat{\i} {\c{S}}eng{\"u}l, \emph{Cutoff for
  conjugacy-invariant random walks on the permutation group}, Probability
  Theory and Related Fields \textbf{173} (2019), 1197--1241.

\bibitem{B98}
B{\'e}la Bollob{\'a}s, \emph{Random graphs}, Springer, 1998.

\bibitem{DS81}
Persi Diaconis and Mehrdad Shahshahani, \emph{Generating a random permutation
  with random transpositions}, Z. Wahrsch. Verw. Gebiete \textbf{57} (1981),
  159--179.

\bibitem{Hou16}
Bob Hough, \emph{The random {$k$} cycle walk on the symmetric group}, Probab.
  Theory Related Fields \textbf{165} (2016), 447--482.

\bibitem{M88}
Peter Matthews, \emph{A strong uniform time for random transpositions}, Journal
  of Theoretical Probability \textbf{1} (1988), 411--423.

\bibitem{N24}
Evita Nestoridi, \emph{Comparing limit profiles of reversible markov chains},
  Electronic Journal of Probability \textbf{29} (2024), 1--14.

\bibitem{NO22}
Evita Nestoridi and Sam Olesker-Taylor, \emph{Limit profiles for reversible
  {M}arkov chains}, Probability Theory and Related Fields \textbf{182} (2022),
  157--188.

\bibitem{Sag13}
Bruce~E Sagan, \emph{The symmetric group: representations, combinatorial
  algorithms, and symmetric functions}, vol. 203, Springer Science \& Business
  Media, 2013.

\bibitem{SZ08}
L~Saloff-Coste and J~Z{\'u}niga, \emph{Refined estimates for some basic random
  walks on the symmetric and alternating groups}, {ALEA} \textbf{4} (2008),
  359--392.

\bibitem{S05}
Oded Schramm, \emph{Compositions of random transpositions}, Israel Journal of
  Mathematics \textbf{147} (2005), 221--243.

\bibitem{Tey20}
Lucas Teyssier, \emph{Limit profile for random transpositions}, Ann. Probab.
  \textbf{48} (2020), 2323--2343.

\bibitem{W19}
Graham White, \emph{A strong stationary time for random transpositions}, arXiv
  preprint arXiv:1910.00770 (2019).

\end{thebibliography}

\end{document}